\documentclass[a4paper,11pt]{article}

\usepackage{amsmath, amscd}
\usepackage{fullpage}

\usepackage[all]{xy}
\usepackage{amsthm}

\usepackage{tikz}
\usetikzlibrary{3d}
\usepackage{amssymb}
\usepackage{latexsym}
\usepackage{enumerate}
\usepackage{graphicx}

\usepackage[utf8]{inputenc}

\newtheorem{Theorem}{Theorem}[section]

\newtheorem{theorem}[Theorem]{Theorem}
\newtheorem{proposition}[Theorem]{Proposition}
\newtheorem{lemma}[Theorem]{Lemma}
\newtheorem{corollary}[Theorem]{Corollary}

\theoremstyle{definition}
\newtheorem{remark}[Theorem]{Remark}
\theoremstyle{definition}
\newtheorem{example}[Theorem]{Example}

\usepackage{comment}
\usepackage{transparent}

\title{Connected components of Isom($\mathbb{H}^3$)-representations of non-orientable surfaces}
\author{Juan Luis Durán Batalla
	\footnote{Supported by doctoral grant BES-2016-079278, also partially supported by the 
		Micinn/FEDER  grant
		PGC2018-095998-B-I00}}
\date{\today}

\begin{document}
\maketitle

\begin{abstract}
Let $N_k$ denote the closed non-orientable surface of genus $k$. In this paper we study the behaviour of the `square map' from the group of isometries of hyperbolic 3-space to the subgroup of orientation preserving isometries. We show that there are $2^{k+1}$ connected components of representations of $\pi_1(N_k)$ in Isom$(\mathbb{H}^3)$, which are distinguished by the Stiefel-Whitney classes of the associated flat bundle. 
\end{abstract}

\section{Introduction}

Let $M$ be a closed surface and $G$ a Lie group. Let $\pi_1(M)$ denote the fundamental group of the surface. The set of homorphisms $\phi: \pi_1(M) \rightarrow G$  is called the variety of representations of the fundamental group of $M$ in $G$ and denoted $\mathrm{hom}(\pi_1(M),G)$. This object can be seen as arising naturally in the context of geometric structures for $M$. For instance, if $M$ is orientable and has a hyperbolic structure, then there is a holonomy map associated to it, $\mathrm{hol}: \pi_1(M) \rightarrow \mathrm{ Isom}^+(\mathbb{H}^2)$, that is, we obtain an element in $\mathrm{hom}(\pi_1(M), G)$, for $G=\mathrm{Isom}^+(\mathbb{H}^2)$ (actually, this map is well-defined up to conjugation by $G$). 

Here we will be interested in studying the case where $M$ is non-orientable and $G=\mathrm{Isom}(\mathbb{H}^3)$. Similar to the previous example, these representations can be related to complex projective structures under the identification $\mathrm{ Isom}(\mathbb{H}^3)\cong \mathrm{PSL}(2,\mathbb{C})\rtimes \mathbb{Z}_2$. Namely, if $M$ is orientable, it is known (see \cite{GKM}) that complex projective structures give  rise to non-elementary representations in $\mathrm{hom}(\pi_1(M), \mathrm{PSL}(2,\mathbb{C}))$ that can be lifted to $\mathrm{SL}(2, \mathbb{C})$ (as we will see later, this can be related to a Stiefel-Whitney class). In the non-orientable case, these representations must be orientation type preserving, that is, the image of an orientable loop must be an orienation preserving isometry and the other way around. Evidently, in general, this must not be satisfied for every representation in $\mathrm{hom}(\pi_1(M), \mathrm{ Isom}(\mathbb{H}^3))$.

The goal of this paper is to find topological invariants which classify the connected components of $\mathrm{hom}(\pi_1, G)$, for $M$ non-orientable and $G=\mathrm{Isom}(\mathbb{H}^3)$. This problem has been studied for a wide variety of groups, we hightlight W. Goldman's paper \cite{Goldman1}, where the author considers the case where $M$ is orientable and $G=\mathrm{PSL}(2,\mathbb{R})$ and $\mathrm{PSL}(2,\mathbb{C})$. Many of the tools and ideas in \cite{Goldman1} are translated to this paper. Regarding non-orientable manifolds, as far as we know, the problem has only been considered in a handful of cases. In \cite{HoLiu1} and \cite{HoLiu2}, a very general case is studied by the authors, but for $G$ connected. In \cite{Xia}, the disconnected group $\mathrm{PGL}(2,\mathbb{R})$ is inspected, however in the case of $M$ an orientable surface.  More closely related are Palesi's papers, $\cite{Palesi1}$ and \cite{Palesi2}, where the $\mathrm{PSL}(2,\mathbb{R})$ and $\mathrm{PGL}(2,\mathbb{R})$ cases are considered, respectively. 
 
To any representation $\phi \in \mathrm{hom}(\pi_1(M), G)$, there is an associated flat $G$-bundle over $M$ (see \cite{MilnorGBundle}). The Stiefel-Whitney classes are a classical invariant of the bundle and can be thought of as invariants of the representation $\phi$. In this sense, these cohomological classes are constant on connected components. In fact, it is enough to use them to distinguish different components and, thus, we obtain that they are indexed by the first and second Stiefel-Whitney classes, giving rise to the following result:

\begin{theorem} Let $N_k$ denote the closed non-orientable surface of genus $k$.
	The representation variety $\mathrm{hom}(\pi_1(N_k), G)$ has $2^{k+1}$ connected components.
\end{theorem}

The paper is structured as follows. In Section~\ref{section:sq_map}, the fibers of the so-called `square map' $[A]\in \mathrm{PSL}(2,\mathbb{C}) \mapsto A^2 \in \mathrm{SL}(2,\mathbb{C})$ are computed, and as a corollary, we have the classification of non-orientable isometries of $\mathbb{H}^3$.  In Section~\ref{section:path_lifting}, we are interested in the path lifting behaviour along the square map and other related maps. Finally, in Section~\ref{section:connected_components}, the results of the previous section are pieced together, obtaining the aforementioned theorem; with a particular emphasis put into the orientation type preserving components.

\section{The square map of non-orientable isometries}
\label{section:sq_map}
Let $G$ denote the group of isometries of hyperbolic $3$-space, $\mathbb{H}^3$. The group $G$ can be identified with 
$$
G=\mathrm{Isom}(\mathbb{H}^3)\cong\mathrm{PSL}(2,\mathbb{C})\rtimes\mathbb{Z}_2
$$
 via the upper half-space model of $\mathbb{H}^3$. We will denote by $G_+$ the connected component of the identity, that is, the subgroup of oriented isometries, $\mathrm{Isom}^+(\mathbb{H}^3)\cong \mathrm{PSL}(2,\mathbb{C})$, and by $G_-$, the subset of nonorientable ones, $\mathrm{Isom}^-(\mathbb{H}^3)$, hence $G=G_-\sqcup G_+$. 
 
The tangent space at any point in $G$ can be identified with its Lie algebra, namely, $\mathfrak{sl}(2,\mathbb{C})\cong \mathfrak{sl}(2,\mathbb{R})\oplus i \mathfrak{sl}(2,\mathbb{R})$. The Lie algebra of $G_+$ is also $\mathfrak{sl}(2,\mathbb{C})$. Regarding $G_-$, we can associate the tangent space at any point with $\mathfrak{sl}(2,\mathbb{C})$ via the inclusion $G_-\hookrightarrow G$, but we can no longer speak of it as the Lie algebra, as $G_-$ is not a group.
 Given an orientation preserving isometry $A \in G_+$, $A_c \in G_-$ will be the composition of $A$ and the complex conjugation $c$.

The universal cover of $G_+$ is the group $\mathrm{SL}(2,\mathbb{C})$ and will be denoted by $\widetilde{G_+}$. There is a natural map we are interested in:
\[
\begin{array}{ccl}
Q:G & \rightarrow & \widetilde{G_+} \\
{[}A] & \mapsto & A^2,
\end{array}
\]
which is well defined as $(\pm A)^2=A^2$. For $A=B_c$, $B\in G_+$, the square map can also be written as $Q(A)=B\overline{B}$, where $\overline{B}$ denotes the complex conjugate of $B$.

The behaviour of $Q$ changes drastically in each connected component $G_-$ and $G_+$. From now onwards, we will interested in its restriction to $G_-$. The following proposition describes the fiber of the square map $Q$.

\begin{proposition}
\label{prop:classification_q1}
Let $B$ be a matrix in $\widetilde{G_+}$ and let us consider the restriction $Q:G_-\to \widetilde{G_+}$. Up to conjugation with an element of $\widetilde{G_+}$, we can assume $B$ is in its Jordan normal form, then
\begin{itemize}
\item If $B=\left( \begin{smallmatrix} \lambda & 0 \\ 0 & \lambda^{-1}
\end{smallmatrix} \right)$, with $\lambda \in \mathbb{R}_+\setminus\{1\}$, the fiber of $B$ is 
$$
Q^{-1}(B)=\{\left( \begin{smallmatrix} a & 0 \\ 0 & a^{-1}
\end{smallmatrix} \right)_c \mid |a|^2=\lambda \}.
$$
\item If  $B =\left( \begin{smallmatrix} e^{i\theta} & 0 \\ 0 & e^{-i\theta}
\end{smallmatrix} \right)$, $\theta \in \mathbb{R}\setminus \{n\pi| n\in \mathbb{Z}\}$, the fiber of $B$ is 
	$$
	Q^{-1}(B)=\{\left( \begin{smallmatrix} 0 & \rho e^{i(\theta+\pi)/2} \\ -\rho^{-1} e^{-i(\theta+\pi)/2} & 0
\end{smallmatrix} \right)_c \mid \rho \in \mathbb{R}^* \}.
$$
\item  If $ B =\left( \begin{smallmatrix} 1 & 1 \\ 0 & 1
\end{smallmatrix} \right) $, the fiber of $B$ is 
$$
Q^{-1}(B)=\{\left( \begin{smallmatrix} 1 & \tau \\ 0 & 1
\end{smallmatrix} \right)_c \mid \mathrm{Re}(\tau)=1/2 \}.$$
\item If $B=Id$, the fiber of $B$ is
$$
Q^{-1}(Id)=\{\left( \begin{smallmatrix} a & b \\ c & \overline{a}
\end{smallmatrix} \right)_c  \mid a\in \mathbb{C}, b,c\in i\mathbb{R}, |a|^2-bc=1 \}=\mathrm{Ad}({G_+}) \left( \begin{smallmatrix}
0 & i \\ i & 0
\end{smallmatrix} \right)_c.
$$
\item If $B=-Id$, the fiber of $B$ is
$$
Q^{-1}(-Id)=\{\left( \begin{smallmatrix} a & b \\ c & -\overline{a}
\end{smallmatrix} \right)_c  \mid a\in \mathbb{C}, b,c\in \mathbb{R}, |a|^2+bc=-1 \}=\mathrm{Ad}(G_+) \left( \begin{smallmatrix}
0 & -1 \\ 1 & 0
\end{smallmatrix} \right)_c .
$$
\item Otherwise, the fiber of $B$ is empty.
\end{itemize}
Finally, if $B$ is not in Jordan form and conjugating it by $g \in \widetilde{G_+}$ takes it to the Jordan form, then the fiber $Q^{-1}(B)$ is computed from the previous cases by conjugating the corresponding fiber by $g^{-1}$.  Furthermore, in every case the fiber is connected.

\end{proposition}

\begin{proof}
We start by noticing the fact that, for $B\in G_+$ with fiber $U=Q^{-1}(B)$, the fiber of a conjugate $\mathrm{Ad}(g)B$, $g\in \widetilde{G_+}$ is $\mathrm{Ad}([g])U$. Moreover, the conjugacy class of any element in $\widetilde{G_+}$ is either diagonal $\left( \begin{smallmatrix} \lambda & 0 \\ 0 & \lambda^{-1} \end{smallmatrix} \right)$ or parabolic $\pm \left(   \begin{smallmatrix} 1 & 1 \\ 0 & 1 \end{smallmatrix} \right)$. Hence, we only have to compute the fibers for these two kinds of matrices. This is straightforward but tedious.

Let $A\in Q^{-1}(B)$, $A=\left(\begin{smallmatrix} a & b \\ c & d \end{smallmatrix} \right)_c$. Then, 
\begin{equation}
\label{Eqn:A2}
	A^2 = \begin{pmatrix} |a|^2+b\overline{c} & a\overline{b}+b\overline{d} \\ c\overline{a}+\overline{c}d & |d|^2+\overline{b}c \end{pmatrix} .
\end{equation}	
	 We will solve for $A^2=B$.

\underline{Case 1:} If $B$ is diagonal, we have from the off-diagonals entries, $a\overline{b}+b\overline{d}=0$ and $\overline{c}a+c\overline{d}=0$. By multipying the first expression by $c$, and the second one by $b$ we get $0=a(2i \mathrm{Im}(c\overline{b}))$. Similarly, we can obtain $0=d(2i\mathrm{Im}(\overline{c}b)).$ Hence, either $\mathrm{Im}(\overline{c}b)=0$ or $\mathrm{Im}(\overline{c}b)\neq 0$ (then, $a=d=0$).

\underline{If $\mathrm{Im}(\overline{c}b) \neq 0:$} $a=d=0$, then the only possibility is $A= \left(\begin{smallmatrix} 0 & b \\ -b^{-1} & 0 \end{smallmatrix} \right)_c$, $B=\left(\begin{smallmatrix} -b/\overline{b} & 0 \\ 0 & -\overline{b}/b \end{smallmatrix} \right)= \left(\begin{smallmatrix} e^{i\theta} & 0 \\ 0 & e^{-i\theta} \end{smallmatrix} \right)$, for some $\theta \in \mathbb{R}$.

\underline{If $\mathrm{Im}(\overline{c}b)=0:$} Then $\lambda \in \mathbb{R}$. We can assume either $a$ or $d$ different from zero as this was already covered. Indeed, it is easy to see that if one of them equal to zero, the other one is zero too. We can write $d=\frac{1+bc}{a}$ and substitute in one of the off-diagonal entries of~\eqref{Eqn:A2} to get $-b=\overline{b}(|a|^2+\overline{c}b)=\overline{b}(\lambda)$. We get a similar equation for $c$. We conclude that either $\lambda=\pm 1$ ($B=Id$) or $b=c=0$.

\begin{itemize}
	\item If $b=c=0$, we have $d=a^{-1}$ and $|a|^2=\lambda$.
	\item If we consider $\lambda=+1$, $B=Id$, $b,c \in i\mathbb{R}$,  from the off-diagonal equations we get $d=\overline{a}$ and it must also be satisfied $|a|^2+b\overline{c}=1$ (which is equivalent to $\mathrm{det}(A)=1$).
	\item If we consider $\lambda=-1$, $B=-Id$, $b,c \in \mathbb{R}$, and, analogously, $d=-\overline{a}$ and $|a|^{2}+bc=-1$ (equivalent again to $\mathrm{det}(A)=1$).
\end{itemize}

Last two cases correspond to the fiber of $\pm Id$. By conjugating first by a traslation $\left(\begin{smallmatrix} 1 & \nu \\ 0 & 1
\end{smallmatrix} \right)$ and then by $\left(\begin{smallmatrix}\mu & 0 \\ 0 & \mu^{-1}
\end{smallmatrix} \right)$ we can see that $Q^{-1}(Id)=\mathrm{Ad}(G_+) \left( \begin{smallmatrix}
0 & i \\ i & 0
\end{smallmatrix} \right)_c $ and $Q^{-1}(-Id)=\mathrm{Ad}(G_+) \left( \begin{smallmatrix}
0 & 1 \\ -1 & 0
\end{smallmatrix} \right)_c$. 

\underline{Case 2:} If $B$ is parabolic, from the diagonal equations we get $b\overline{c}\in \mathbb{R}$ and therefore, multiplying the off-diagonal equation $a\overline{b}+b\overline{d}=\pm 1$ by $c$ we obtain $c=0$. Hence, $d=a^{-1}$ and $|a|=1$. Finally, by writing $a$ and $b$ in polar form and focusing on the off-diagonal equation, we obtain $a=\pm 1$ and $\mathrm{Re} \ b=\pm 1/2$.

\end{proof}

\begin{remark}
\label{rk:def-hyp-ell-par}
Let $A\in G_-$ such that $Q(A)\neq Id$. We will say $A$ is \emph{hyperbolic}, \emph{elliptic} or \emph{parabolic} according to $Q(A)$ being hyperbolic, elliptic or parabolic, respectively.
\end{remark}

\begin{corollary} The image of $Q$ is 
	$$
	\mathcal{J}:=\{A\in \widetilde{G_+}\mid \mathrm{tr}(A)\in (-2,\infty) \}\cup \{-Id \},
	$$
	 where $\mathrm{tr}(A)$ denotes the trace of $A$.
\end{corollary}

\begin{remark}
Proprosition $\ref{prop:classification_q1}$ can be extended to the quotient map $[Q]: G_-\mapsto G_+$. Then, $Q^{-1}{[Id]}$ and $Q^{-1}([\left( \begin{smallmatrix} e^{i\theta} & 0 \\ 0 & e^{-i\theta}
\end{smallmatrix} \right) ]) $ have two connected components whereas $Q^{-1}([\left( \begin{smallmatrix} \lambda & 0 \\ 0 & \lambda^{-1}
	\end{smallmatrix} \right)])$ and $Q^{-1}([\left( \begin{smallmatrix} 1 & \tau \\ 0 & 1
	\end{smallmatrix} \right)])$ just one.
\end{remark}

\begin{corollary}[Classification of non-orientable isometries of $\mathbb{H}^3$]
Let us consider a non-orientable isometry of hyperbolic $3$-space. Then, up to conjugation, it is one of the following:
\begin{itemize}
	\item Composition of a reflection on a hyperplane with hyperbolic translation in an axis contained in said hyperplane.
	\item Composition of a reflection on a hyperplane with a rotation in an axis perpendicular to the aforementioned hyperplane.
	\item Composition of a reflection on a hyperplane with a parabolic transformation with fixed point an ideal point of the hyperplane.
	\item Reflection on a hyperplane. 
	\item Inversion through a point. 
\end{itemize}
\end{corollary}

\begin{proof} Interprete each case of proposition~\ref{prop:classification_q1}.
\end{proof}

\bigskip

Let $N_k$ be the closed non-orientable surface of (non-orientable) genus $k$ and $\pi_1(N_k)$ its fundamental group. Then, $\pi_1(N_k)$ admits a representation $\langle a_1, \dots, a_k | a_1^2 \cdots a_k^2=1\rangle$. Let $\mathrm{hom}(\pi_1(N_k), G)$ be the representation variety, which can be identified with the algebraic set 
$$
\{A_1, \dots, A_k \in G \mid A_1^2\cdots A_k^2=[Id] \}.
$$

For $\phi\in \mathrm{hom}(\pi_1(N), G)$, we say that $\phi$ \emph{preserves the orientation type} if it satisfies $\phi(\gamma)\in G_-$ if and only $\gamma$ is represented by a loop reversing the orientation.

We can state two inmediate corollaries of Proposition~\ref{prop:classification_q1} regarding representation varieties of genus $1$ and $2$. The general case will be relegated to the last section.

\begin{corollary}
Let $N_1$ be a projective plane. The variety of orientation type preserving representations has two connected components.
\end{corollary}

\begin{corollary}
	\label{co:connected_klein}
Let $N_2$ be a Klein bottle. The variety of orientation type preserving representations has two connected components.
\end{corollary}
\begin{proof}
Let $A,B \in G_-$ satisfy $A^2=B^2$. Then, in $\widetilde{G_+}$, $Q(A)=\pm Q(B)$. If the sign is plus, $A$ and $B$ are in the same fiber of $Q$ and which is connected by Proposition~\ref{prop:classification_q1}. Thus, there is a path connecting $A$ and $B$ inside the fiber. Moreover, as $G_-$ is connected, any two representations $(A_1,B_1)$, $(A_2,B_2)$ with $A_i=B_i$ and in different fibers can be joined by a path. Otherwise if $Q(A)=-Q(B)$, then $Q(A)$ is either elliptic or $\pm Id$. By connectedness of both the fibers of $Q$ and the subset of elements of $G_-$ which are neither hyperbolic nor parabolic, we can prove in a similar fashion that the subset of representations such that $Q(A)=-Q(B)$ is connected too.
\end{proof}

\section{Path-lifting of the square map}

\label{section:path_lifting}

Dealing with connected components of representation varieties is easier if we switch the approach from connectedness to path-connectedness. As noted in~\cite{Goldman1} in the frame of representation varieties they are equivalent. A very useful tool in this regard is \emph{path-lifting property}. We say that a map $f:X\rightarrow Y$ satisfies the path-lifting property if and only if for every point $x\in X$ and every path $\gamma:[0,1]\rightarrow \mathrm{ Im}\, f\subset Y$ with $\gamma(0)=f(x)$, there exists,  up to reparametrization of $\gamma$, a lift of $\gamma$ to a path $\sigma:[0,1]\mapsto [0,1]$ with $\sigma(0)=x$. Notice that, in general, the path lifting property does not imply uniqueness of the lift.

From the submersion normal form, we can prove:

\begin{lemma}
	\label{lm:path-lifting}
	 Let $f:X\rightarrow Y$ be a smooth map between smooth manifolds. If $f$ is a submersion, then it satisfies the path lifting property.
\end{lemma} 

An inmediate consequence of satisfying the path-lifting property is the following one: let $f$ satisfy the path-lifting property and let $\mathrm{ Im}\,f$ be connected. If there exists a point $y\in \mathrm{ Im}\, f$ whose fiber is path-connected, then the domain is connected.

We will denote  $\mathcal{J}_0:=\mathcal{J}\setminus \{\pm Id \}$. The map $Q$ restricted to $Q^{-1}(\mathcal{J}_0)$ has good properties:

\begin{lemma}
	\label{lm:submersion_square}
	The map $Q$ restricted to $Q^{-1}(\mathcal{J}_0)$ is a submersion. Furthermore, it satisfies the path lifting property.
\end{lemma}

\begin{proof} The differential of $Q$ at $A_c\in G_-$,  applied to a tangent vector $\xi$ is
	\begin{equation*}
	dQ(\xi)=A\xi \overline{A}+A\overline{A}\overline{\xi}=(A\overline{\xi}A^{-1}+\xi)A\overline{A},
	\end{equation*}
	where we are taking multiplication at right. Thus, from the Lie algebra point of view, it is
	\[
	\begin{array}{ccc}
	\mathfrak{sl}(2,\mathbb{C}) & \longrightarrow &\mathfrak{sl}(2,\mathbb{C}) \\
	\xi & \longmapsto & A\overline{\xi}A^{-1}+\xi \; .
	\end{array} \]
	As we are only interested on the rank of the map, we can swap $\xi$ by its conjugate $\overline{\xi}$, which leaves the image as $\mathrm{Ad}(A)\xi+\overline{\xi}$. Similarly, we can take $A$ to be any element in its conjugacy class.
	
	Thus, the proof will come from computing the adjoint representation for each conjugacy class of element in $G_-\setminus Q^{-1}(\pm Id)$. Proposition~\ref{prop:classification_q1} states that, up to conjugation, the element $A_c$ can be assumed to be either hyperbolic, elliptic or parabolic (see remark~\ref{rk:def-hyp-ell-par}), that is, either
	$$
	\begin{pmatrix}
	\lambda & 0 \\
	0 & \lambda^{-1}
	\end{pmatrix}_c , \; \; 
	\begin{pmatrix}
	0 & e^{i(\theta+\pi)/2} \\
	-e^{-i(\theta+\pi)/2} & 0
	\end{pmatrix}_c , \;  \mathrm{ or} \; \;
	\begin{pmatrix}
	1 & 1 \\
	0 & 1
	\end{pmatrix}_c,
	$$
	where $\lambda\in \mathbb{R}$, $\theta \in (0,\pi)$. We will call each respective case hyperbolic, elliptic and parabolic and denote the matrix $A\in G_+$ by $A_{hyp}$, $A_{ell}$ or $A_{par}$. If $\xi=\left(\begin{smallmatrix}
	x_3 & x_1 \\
	x_2 & -x_3
	\end{smallmatrix}  \right)$ belongs to the lie algebra $\mathfrak{sl}(2,\mathbb{C})$, then, the adjoint for each case is:
	$$
	\mathrm{Ad}(A_{hyp})\xi= \left( \begin{smallmatrix}
	x_3& \lambda^2 x_1 \\
	\lambda^{-2}x_2 & -x_3
	\end{smallmatrix} \right)
	, \; \; 
	\mathrm{Ad}(A_{ell})\xi= \left( \begin{smallmatrix}
	-x_3& x_2 e^{i\theta} \\
	x_1e^{-i\theta} & x_3
	\end{smallmatrix} \right)
	, \; \; 
	\mathrm{Ad}(A_{par})\xi=\left( \begin{smallmatrix}
	x_2+ x_3& x_1-x_2-2x_3 \\
	x_2 &-x_2 -x_3
	\end{smallmatrix} \right)
	.
	$$
	Thus, as an action of $\mathrm{ SO}(2,1)$, the adjoint representation is, respectively
	\[
	\mathrm{Ad}(A_{hyp}) =
	\begin{pmatrix}
	\lambda^2 & 0 & 0 \\
	0 & \lambda^{-2} & 0 \\
	0 & 0 & 1
	\end{pmatrix}, \;
	\mathrm{Ad}(A_{ell}) =
	\begin{pmatrix}
	0 & e^{i\theta} & 0 \\
	e^{-i\theta} & 0 & 0 \\
	0 & 0 & -1
	\end{pmatrix}, \;
	\mathrm{Ad}(A_{par}) =
	\begin{pmatrix}
	1 & -1 & -2 \\
	0 &  1& 0 \\
	0 & 1 & 1
	\end{pmatrix}.
	\]
	We are interested in writing the lie algebra as $\mathfrak{sl}(2,\mathbb{C})=\mathfrak{sl}(2,\mathbb{R})\oplus i \mathfrak{sl}(2,\mathbb{R})$, hence, the matrix of the linear map $\xi \mapsto \mathrm{Ad}(A)\xi + \overline{\xi}$ is:
	\begin{align*}
	\mathrm{Ad}(A_{hyp})\xi + \overline{\xi} & =\begin{pmatrix}
	\lambda^2+1 & 0 & 0 & 0 & 0 & 0 \\
	0 & \lambda^{-2}+1 & 0 & 0 & 0 & 0 \\
	0 & 0 & 2 & 0 & 0 & 0 \\
	0 & 0 & 0 & \lambda^2-1 & 0 & 0 \\
	0 & 0 & 0 & 0 & \lambda^{-2}-1 & 0 \\
	0 & 0 & 0 & 0 & 0 & 0 
	\end{pmatrix}, \\
	\mathrm{Ad}(A_{ell})\xi + \overline{\xi} & =\begin{pmatrix}
	1 & \cos \theta & 0 & 0 & -\sin \theta & 0 \\
	\cos \theta & 1 & 0 & \sin \theta & 0 & 0 \\
	0 & 0 & 0 & 0 & 0 & 0 \\
	0 & \sin \theta & 0 & -1 & \cos \theta & 0 \\
	- \sin \theta & 0 & 0 & \cos \theta & -1 & 0 \\
	0 & 0 & 0 & 0 & 0 & -2 
	\end{pmatrix},
	\\
	\mathrm{Ad}(A_{par})\xi + \overline{\xi} & =\begin{pmatrix}
	2 & -1 & -2 & 0 & 0 & 0 \\
	0 & 2 & 0 & 0 & 0 & 0 \\
	0 & 1 & 2 & 0 & 0 & 0 \\
	0 & 0 & 0 & 0 & -1 & -2 \\
	0 & 0 & 0 & 0 & 0 & 0 \\
	0 & 0 & 0 & 0 & 1 & 0 
	\end{pmatrix}.
	\end{align*}
	We notice that the rank is always five (in the elliptic case, due to $\theta\neq0,\pi$) , which equals the dimension of the image.
	
	Finally, it satisfies the path-lifting property due to lemma~\ref{lm:path-lifting}.
\end{proof}

If we try to extend Lemma~\ref{lm:submersion_square} to the whole image $\mathcal{J}$ we are bound to fail. A geometric interpretation of why this points are troublesome comes from noticing that a rotation in $S^2$ is given by its unique axis and its angle of rotation. Hence, the `square root' will have the same axis and half the angle. We can consider a sequence of rotations such that the angle goes towards $0$ or $\pi$ (where there is no longer a unique axis) but the sequence of axes does not converge. Thus, the sequence has a limit but the square root will not. This is illustrated in the following example (\cite{Palesi1}):

\begin{example}
	Let
	\[
	g_t=\begin{pmatrix}
	\sqrt{2}+\sin (1/t) & \cos (1/t) \\
	\cos (1/t) & \sqrt{2}- \sin (1/t)
	\end{pmatrix}, \; \;
	R_\theta = \begin{pmatrix}
	e^{i\theta} & 0 \\
	0 & e^{-i\theta}
	\end{pmatrix}.
	\]
	
	In general, $\mathrm{Ad}(g)R_{\theta_t}$ tends towards to $\pm Id$ if we make $\theta_t$ tend to $0$ or $\pi$, respectively. Thus, in the particular case $g_tR_{\theta_t} g_t^{-1}$ with $\theta_t=\pi -t$, we obtain
	
	\[
	g_t R_{\theta_t} g_t^{-1} {\xrightarrow{t\to 0}}-Id. 
	\]
	
	On the other hand, the path $g_t R_{\theta_t} g_t^{-1}$ cannot be lifted along $Q$, due to the appearance of the terms $\sin(1/t)$ and $\cos(1/t)$ in any possible lift  of the path outside of $0$.
	
	Same example works with $\theta_t=t$, where the limit now is the identity.
	
\end{example}

Let $X(F_2,\widetilde{G_+})$ denote the variety of characters of $F_2$, the free group on two elements and let 
\begin{equation}
\label{Eqn:Character_map}
\begin{array}{ccc}
	\chi: \widetilde{G_+}\times \widetilde{G_+} & \longrightarrow & X(F_2,\widetilde{G_+})\simeq \mathbb{C}^3 \\
	(A,B) & \longmapsto & \chi(A,B):=(\mathrm{tr}\,A,\mathrm{tr}\,B,\mathrm{tr}\,AB)
\end{array}
\end{equation}
 be the character map, where $\mathrm{tr} \, A$ denotes the trace of $A$. The polinomial map $\kappa:\mathbb{C}^3  \rightarrow  \mathbb{C}$, $\kappa(x,y,z):=x^2+y^2+z^2-xyz-2$
  satisfies $\kappa (\chi(a,b))=\mathrm{tr}\,[a,b]$, where $[\cdot, \cdot]$ is the commutator. 
  
  \bigskip
 
Let 
\[
\begin{array}{ccl}
Q_n: G^n & \longrightarrow & \widetilde{G_+}^n \\
(A_1, \dots, A_n) & \longmapsto & (Q(A_1), \dots, Q(A_n)).
\end{array}
\]
Thus, $Q_1$ coincides with $Q$. The map $Q_n$ can be extended to the varieties of representations. For instance, let $a,b,c$ be generators of $\pi_1(N_3)$, then the map $Q_3$ in the variety of representations $\mathrm{hom}(\pi_1(N_3), G)$ is
\[
\begin{array}{ccl}
Q_3: \mathrm{hom}(\pi_1(N_3),G) & \longrightarrow  & \{ (X,Y,Z) \in (\widetilde{G_+})^3  | XYZ=\pm Id \}, \\
\phi & \longmapsto & (Q(\phi(a)), Q(\phi(b)), Q(\phi(c)))
\end{array}
\]

We will often use the following notation: let $\phi(a)=A_c$ where $A\in G_+$ and $c$ is the complex conjugation and, analogously, $\phi(b)=B_c$, $\phi(c)=C_c$. Then, $Q(\phi(a))=A\overline{A}$, where $\overline{A}$ denotes the complex conjugated matrix. Therefore, the map $Q_3$ can be written as
$$
(A_c, B_c, C_c) \mapsto (A\overline{A}, B\overline{B}, C\overline{C}).
$$

Moreover, the image of $Q_3$ in $\{ (X,Y,Z) \in (\widetilde{G_+})^3 | XYZ=\pm Id \}$ is identified with 
$$
\mathrm{ Im}\, Q_3 \cong \{ (X,Y) \in \mathcal{J}\times \mathcal{J}   | XY\in \mathcal{J} \}.
$$ 

Finally, let $S:=\{(X,Y) \in \mathcal{J}_0 \times \mathcal{J}_0\mid [X,Y]\neq Id, \; XY \in \mathcal{J}_0  \}$.

The following lemma can be found in~\cite{Goldman1}:

\begin{lemma}
	\label{lm:differential_chi_exhaustive}
	Let $(A,B)\in \widetilde{G_+}\times \widetilde{G_+}$. Then, the differential of $\chi$ (cf.~\eqref{Eqn:Character_map}) at $(A,B)$, $d_{(A,B)}\chi$ is surjective if and only $A$ and $B$ do not commute.
\end{lemma}

\begin{corollary}
	\label{Coro:path-lifting-chiq}
	Let $(A,B) \in G_-\times G_-$. The differential of $\chi \circ Q_2$ at $(A,B)$, $d_{(A,B)}\chi \circ Q_2$, is surjective if and only if $Q(A)$ and $Q(B)$ do not commute. In particular, the map $\chi \circ Q_2$ satisfies the path-lifting property.
\end{corollary}

\begin{proof}
	From Lemma~\ref{lm:differential_chi_exhaustive} we see that a necessary condition for the differential to be exhaustive is that $Q(A)$ and $Q(B)$ do not commute.
	
	In the other direction, if $Q(A)$ and $Q(B)$ do not commute, in particular they are different from $\pm Id$, then, from Lemma~\ref{lm:submersion_square}, $(A,B)\mapsto (Q(A),Q(B))$ is a submersion at $(A,B)$. By Lemma~\ref{lm:differential_chi_exhaustive}, $\chi \circ Q_2$ is a submersion too.
	The last assertion is a consequence of Lemma~\ref{lm:path-lifting}.
\end{proof}

\begin{lemma}
	\label{lm:product_submersion}
	The set of regular points of the map $G_-\times G_- \rightarrow \widetilde{G_+}$, $A,B\mapsto Q(A)Q(B)$ is 
	$$\{A, B| [Q(A),Q(B)] \neq Id \} \cup \{A,B \mid (\mathrm{tr}(Q(A))-2)(\mathrm{tr}(Q(B))-2)<0 \}.
	$$
	More generally, $A_1, \cdots, A_n$ is a regular point of the map $\mathrm{prod} \circ Q$, where \allowbreak $\mathrm{prod}:\widetilde{G_+}^n\rightarrow \widetilde{G_+}$ denotes the product, if and only if it exists $i,j\in \{1,\dots, n \}$ such that either $Q(A_i), Q(A_j)$ do not commute or $(\mathrm{tr}(Q(A_i))-2)(\mathrm{tr}(Q(A_j))-2)<0$.
\end{lemma}

\begin{proof}
	Let us asssume right multiplication in the Lie group. A straightforward computation shows that the differential applied to a tangent vector $(\xi, \eta)$ is
	\begin{equation*}
	\xi A\overline{A}B\overline{B}+ A \overline{\xi} \overline{A} B\overline{B}+A\overline{A}\eta B \overline{B}+A\overline{A}B\overline{\eta}\overline{B}.
	\end{equation*}
	This corresponds to the vector of the Lie algebra $\mathfrak{sl}(2,\mathbb{C})$
	\begin{equation*}
	\xi+\mathrm{Ad}(A)\overline{\xi}+\mathrm{Ad}(A\overline{A})\eta +\mathrm{Ad}(A\overline{A}B)\overline{\eta}.
	\end{equation*}
	We can multiply the expression by $\mathrm{Ad}(A\overline{A})^{-1}$ and swap $\xi$ by $\mathrm{Ad}(A)^{-1}\xi$, and $\eta$, by $\overline{\eta}$. We obtain
	\begin{equation*}
	\mathrm{Ad}(\overline{A})^{-1}(\xi)+\overline{\xi}+\mathrm{Ad}(B)(\eta) +\overline{\eta}.
	\end{equation*}
	We can assume $A$ to be in its Jordan normal form so that its adjoint representation of $A$ is easy to compute. On the other hand, with the previous assumption, we will have no control on the adjoint representation of $B$, so we need to compute the adjoint representation for any matrix $X=\left( \begin{smallmatrix} a & b \\ c & d \end{smallmatrix} \right)$ applied to an element $\xi=\left( \begin{smallmatrix}
	x_3 & x_1 \\
	x_2 & - x_3
	\end{smallmatrix} \right)$:
	\begin{equation*}
	\mathrm{Ad}(X)(\xi)= \begin{pmatrix}
	(ad+bc) x_3 + bd x_2 - ac x_1 & - 2ac x_3 - b^2x_2+a^2x_1 \\
	2cdx_3 + d^2x_2-c^2y_1 &  -((ad+bc) x_3 + bd x_2 - ac x_1 ) 
	\end{pmatrix}.
	\end{equation*}
	Thus,
	\begin{equation*}
	\mathrm{Ad}(X)(\xi)+\overline{\xi}= \begin{pmatrix}
	(ad+bc) x_3+\overline{x_3} + bd x_2 - ac x_1 & - 2ac x_3 - b^2x_2+a^2x_1+\overline{x_1} \\
	2cdx_3 + d^2x_2+\overline{x_2}-c^2x_1 &  -((ad+bc) x_3+\overline{x_3} + bd x_2 - ac x_1 ). 
	\end{pmatrix}
	\end{equation*}
	Taking  $A$ in its Jordan form and resting upon the computation in the proof of Lemma~\ref{lm:submersion_square}, we prove the first assertion.
	
	The general statement follows in a similar way. The differential applied to a tangent vector $(\xi_1, \cdots, \xi_n)$ is
	\begin{equation*}
	\sum_{i=1}^n \left( \mathrm{Ad}(\prod_{j=1}^{i-1} A_j\overline{A_j})(\xi_i+\mathrm{Ad}(A_i)\overline{\xi_i}) \right).
	\end{equation*}
The image of each summand has rank five as seen in the proof of Lemma~\ref{lm:submersion_square}, therefore, $(A_1,\dots, A_n)$ is a regular point iff there exists $i<j$ such that the respective summands generate the whole space. Let us apply induction on the distance between both matrices $k:=j-i$.

For $k=1$, the case $n=2$ can be applied to conclude that either  $Q(A_i)$ and $Q(A_{i+1})$ do not commute or $(\mathrm{tr}(Q(A_i))-2)(\mathrm{tr}(Q(A_{i+1}))-2)<0$. If $k>1$, either the image of the tangent vectors $(0,\dots, 0, \xi_i, \xi_{i+1}, 0, \dots, 0)$ generate the whole tangent space or not. If they generate it, then we can take $A_i$ and $A_{i+1}$ instead. Otherwise, the tangent vectors associated to $A_{i+1}$ and $A_j$ generate the whole image, and we can apply the induction hypothesis.
 \end{proof}

\section{Representation varieties}

\label{section:connected_components}

In this section we apply the results of the previous section to compute the connected components of the variety of representations $\mathrm{hom}(\pi_1(N_k), G)$. They will be indexed by two Stiefel-Whitney classes. The first of them is due to the two connected components of $G$, where as the second one is the second Stiefel-Whitney class of the associated flat principal bundle. We first focus on computing the second Stiefel-Whitney class in the case of orientation type preserving representations, whose study is a little bit more involved. In order to compute the number of connected components, we will compute first them for non-orientable genera $k=2, 3$ and then, the general case by induction by cutting the surface in a subsurface of genus $k-2$ and another one of genus $2$.

\subsection{The orientation type  preserving components}
\label{subsection:orientation_preserving}

We will compute here the connected components of the space of \emph{orientation type  preserving} representations of the fundamental group of a closed non-orientable surface $N_k$ into $G_-$. We will denote the set of orientation type preserving representations by
$$
\mathrm{hom}^{tp} (\pi_1(N_k), G).
$$ 

These connected components can be identified as fibers of the Stiefel-Whitney map $w_2: \mathrm{hom}^{tp}(\pi_1(N_k), G)\rightarrow \mathbb{Z}_2$. The algebraic variety can be identified with the set
$$
\mathrm{hom}^{tp}(\pi_1(N_k), G)=\{A_1,\dots, A_k\in G_- \mid [Id]=R(A_1,\dots, A_k)=\pi( \prod_{i=1}^k Q(A_i) ) \},
$$
where $\pi: \widetilde{G_+} \rightarrow G_+$ is the covering projection. The relator map $R$ can be lifted to $\tilde{R}: G_-^k\rightarrow \widetilde{G_+}$ as $\tilde{R}=\prod Q(A_i)$ and its image lies on the set $\{\pm Id \}$. This lifted relator map $\tilde{R}$ is constant on connected components and its image can be identified with the second Stiefel-Whitney class of the associated flat $G_+-bundle$ (see \cite{MilnorGBundle}). 

Let $C\in \widetilde{G_+}$, $n\geq 2$ and let us define 
$$
X_n(C):=\{(A_1,\dots A_n) \in (G_-)^n \mid \prod_{i=1}^n Q(A_i)=C \}.
$$

The set $X_n(Id)$ is precisely the variety of representations $\mathrm{hom}^{tp}(\pi_1(N_n), \widetilde{G_+})$. The whole representation variety \allowbreak $\mathrm{hom}^{tp}(\pi_1(N_n), {G})$ can be identified with the set $X_n(Id)\cup X_n(-Id)$. Each set $ X_n((-Id)^u)$ corresponds to the preimage of the second Stiefel-Whitney class, $w_2^{-1}(u),$ $u\in \mathbb{Z}_2$, of the  principal bundle. If we prove that $X_n((-Id)^u)$ is non-empty and connected for $u=\pm 1$, then $\mathrm{hom}^{tp}(\pi_1(N_n), {G})$ has two connected components.

Lemma~\ref{lm:product_submersion} shows that instead of working with the whole space $X_n(C)$ it is actually more practical to work with
\begin{equation*}
X'_n(C):=\{(A_1,\dots, A_n) \in X_n(C) \mid \exists i,j \textrm{ such that }[Q(A_i),Q(A_j)]\neq Id   \}.
\end{equation*}
In fact, for us it will be more useful to restrict to the following subset, for $n\geq 4$,
\begin{equation*}
X''_n(C):=\{(A_1,\dots, A_n) \in X_n(C) \mid \exists i,j\leq n-2 \textrm{ such that }[Q(A_i),Q(A_j)]\neq Id   \}.
\end{equation*}
Notice that $X'_2(\pm Id)=\emptyset$, $X'_2(C)=X_2(C)$ if $C\neq \pm Id$. For $n=1$, let us define $X_1(C):=Q^{-1}(C)$. Last, for $n=3$, we define
$$
X''_3(C):=\{ (A_1,A_2,A_3) \in X_3(C) \mid Q(A_1)\neq \pm Id, \, [Q(A_2),Q(A_3)]\neq Id \}.
$$
As the following lemma shows, from a connectivity point of view it is indifferent considering either $X_n(C)$, $X'_n(C)$ or $X''_n(C)$.

\begin{remark}
\label{rk:inclusion_x(c)}
	For $n\geq 3$, $X''_n(C)\subset X'_n(C) \subset X(C)$. Thus, if $X''_n(C)$ is dense in $X_n(C)$ and connected, then $X'_n(C)$ is dense and connected too.
\end{remark}

\begin{lemma}
	\label{lm:x'_dense}
	Both sets $X'_n(C)$ and $X''_n(C)$ are dense in $X_n(C)$, when defined, for any $C\in \widetilde{G_+}$, $n\geq 3$.
\end{lemma}

\begin{proof}
	Let $(A_1, \dots, A_n)\in X'_n(C)$, we will show that we can find elements as close as wanted to $(Q(A_1), \dots, Q(A_n))\in \mathcal{J}^n$. Then, there will be suitable preimages in $(G_-)^n$ as closed as desired to $(A_1, \dots, A_n)\in X'_n(C)$. The proof will apply for both $X'_n(C)$ and $X''_n(C)$ due to remark~\ref{rk:inclusion_x(c)}. The main idea consists of swapping two consecutive elements $A_i, A_{i+1}$ for two close enough, non-commuting elements $A_i', A_{i+1}'$ so that the product remains the same, $A_i A_{i+1}=A_i'A_{i+1}'$.
	
	Let all of the $Q(A_i)$ commute, then either all of them can be conjugated to a diagonal matrix or a parabolic one. Let us suppose that all of them are diagonal and let, $Q(A_1)=\left( \begin{smallmatrix}
	\lambda & 0 \\
	0 & \lambda^{-1}
	\end{smallmatrix} \right)$, $Q(A_2)=\left( \begin{smallmatrix}
	\mu & 0 \\
	0 & \mu^{-1}
	\end{smallmatrix} \right)$, then if one of the matrix is not the inverse of the other ($\mu\neq \pm \lambda^{-1}$), we can take $\epsilon$, $\delta$ as small as wanted so that $B_1=\left( \begin{smallmatrix}
	\lambda & \epsilon \\
	0 & \lambda^{-1}
	\end{smallmatrix} \right)$, $B_2=\left( \begin{smallmatrix}
	\mu & \delta \\
	0 & \mu^{-1}
	\end{smallmatrix} \right)$ do not commute and $B_1B_2=Q(A_1)Q(A_2)$. If there are no two consecutive elements $A_i, A_{i+1}$ such that $Q(A_i)Q(A_{i+1})\neq Id$, then  if we take $B_1, B_2$ as before, the elements $B_2$ and $Q(A_3)$ will not commute. In $X''_n(C)$ the same works if we start with $Q(A_2)$ and $Q(A_3)$.
	
	In the parabolic case, given two matrices $Q(A_i)=\left( \begin{smallmatrix}
	1 & x_i \\
	0 & 1
	\end{smallmatrix} \right)$, $i=1,2$ (we can assume $\mathrm{tr}\; Q(A_i)=2$) matrices $B_i=\left( \begin{smallmatrix}
	\lambda & y_i \\
	0 & \lambda^{-1}
	\end{smallmatrix} \right)$ can be chosen as close as wanted to $Q(A_i)$ such that they do not commute and $B_1B_2=Q(A_1)Q(A_2)$, as long as $Q(A_1)Q(A_2)\neq Id$. If they are inverse matrices, after deforming them as before, $B_2$ and $Q(A_3)$ will no longer commute.  
	
	The case where all of the $Q(A_i)$ are $\pm Id$ can be obtained by any of the two other cases.

\end{proof}

\begin{proposition}
	\label{pp:x'3_connected}
	The sets $X'_3(Id)$ and $X'_3(-Id)$ are non-empty and connected.
\end{proposition}

\begin{proof}
	Let $(A,B,C)\in X'_3( Id)$, we can assume without loss of generality that $A$ and $B$ satisfy $[Q(A),Q(B)]\neq Id$. Let us consider the case $X'_3(Id)$. Let us fix $(x_1,y_1,z_1)\in (-2,\infty)^3$ such that $\kappa(x_1,y_1,z_1)\neq 2$, then $\kappa^{-1}(x_1,y_1,z_1)$ is composed of a single  $\widetilde{G_+}$-orbit. Let $(A_1,B_1)\in S=\{(X,Y) \in \mathcal{J}_0 \times \mathcal{J}_0\mid [X,Y]\neq Id, \; XY \in \mathcal{J}_0  \}$ such that $(Q(A_1),Q(B_1))$ is in said conjugacy class.
	
	We will construct a path from $(A,B)$ to $(A_1, B_1)$ inside of $X'_3(Id)$. Let  $(x_0,y_0,z_0):=\allowbreak \chi (Q(A),Q(B))$, then a path can be constructed in $(-2,\infty)^3$ joining $(x_0,y_0,z_0)$ and $(x_1,y_1,z_1)$. By Corollary~\ref{Coro:path-lifting-chiq} The path can be lifted to $S$ starting at $(A_0,B_0)$ and ending at the fiber of $(x_1,y_1,z_1)$ and, as the fiber is connected, it can be continued to $(A_1,B_1)$.
	
	The case $X'_3(-Id)$ can be proven in the same way choosing $(x_1,y_1,z_1) \in (-2,\infty)^2\times(-\infty,2)$ instead.
\end{proof}

\begin{corollary}
	\label{co:x3_connected}
	The sets $X_3( Id)$ and $X_3(-Id)$ are non-empty and connected.
\end{corollary}

\begin{proof}
	Apply Lemma~\ref{lm:x'_dense} to Proposition~\ref{pp:x'3_connected}.
\end{proof}

\begin{proposition}
	\label{pp:x2(c)}
	The set $X_2(C)$ is non-empty and connected for any $C\in \widetilde{G_+}$.
\end{proposition}
\begin{proof}
	The case $C=\pm Id$ is Corollary~\ref{co:connected_klein}.
	
	Let $z=\mathrm{tr} \ C$ and let us consider $(x_1,y_1,z)\in (-2,\infty)^2 \times \mathbb{C}$ such that $\kappa(x_1,y_1,z)\neq 2$. The fiber $\kappa^{-1}(x_1,y_1,z)$ is one $\widetilde{G_+}$-conjugacy class (see \cite{Goldman1}). Let us fix some element $(A_1,B_1) \in G_-\times G_-$ in the aforementioned conjugacy class such that $Q(A_1)Q(B_1)=C$. Now, for any $(A_0,B_0)\in X'_2(C)$, with $\chi(Q(A_0),Q(B_0))=(x_0,y_0,z)$ a path $(x_t,y_t,z)$ can be constructed in $(-2,\infty)^2\times \mathbb{C}$. The path can be reparametrized and lifted to $\{(X,Y)\in \widetilde{G_+}\times \widetilde{G_+} \mid [X,Y]\neq Id \}$ starting at $(Q(A_0),Q(B_0))$ and, as $Q$ is a submersion in $\{(X,Y)\in \widetilde{G_+}\times \widetilde{G_+} \mid [X,Y]\neq Id \}$ (see Lemma~\ref{lm:submersion_square}), it can be lifted to $G_-\times G_-$; notice, however, the resulting path $(X_t,Y_t,C_t)$ does not necessarily satisfy $C_t=C$. 
	
	We can obtain continuously a path $g_t\in \widetilde{G_+}$ such that $C_t=g_tC g_t^{-1}$, so conjugating by $g_t^{-1}$ we obtain a path in $X'_2(C)$. This can be done by writing the matrix $C$ in its Jordan canonical form and understanding $g_t$ as a change of basis matrix to its Jordan form. Due to the fact that the Jordan form of $C_t$ remains constant during the whole path (otherwise, it wouldn't be true), this basis can be chosen to depend continously on $C_t$: for instance, we can ask for the basis to have constant norm and first coordinate real.  
	
	The path we have thus constructed ends in the fiber $\kappa^{-1}(x_1,y_1,z)$, which is connected (it is a $\widetilde{G_+}$-conjugacy class). The elements of the fiber having $C$ as third coordinate is a $\mathrm{Stab}(C)$-conjugacy class, where $\mathrm{Stab}(C)$ denotes the stabilizer of $C$. The stabilizer is connected unless $C$ is parabolic, then it has two connected components $\mathrm{Stab}^0(C)$ (the connected component of the identity) and $-\mathrm{Stab}^0(C)$. Therefore, in any case, the $\mathrm{Stab}(C)$-conjugacy class is connected. Thus, the path can be joined with $(A_1, B_1)$.

\end{proof}

Let us denote by $f_{n-2}$ the map $f_{n-2}:X_n(C)\rightarrow \widetilde{G_+}$ defined by $f_{n-2}(A_1,\dots, A_n):=Q(A_1)\cdots Q(A_{n-2})$.
\begin{proposition}
	The set $X''_n(C)$ is non-empty and connected for any $C\in \widetilde{G_+}$ and $n\geq 3$. Moreover, the path $f_{n-2}$ satisfies the path lifting property.
\end{proposition}

\begin{proof}
	Let us work by induction on $n$. We need then two initial cases, one for $n$ even and another for $n$ odd.

The fiber of $f_1:X''_3(C)\rightarrow \mathcal{J}_0$ at $\nu\in \mathcal{J}_0$ is  $Q^{-1}(\nu)\times X_2'(\nu^{-1}C)$. By Propositions~\ref{prop:classification_q1} and ~\ref{pp:x2(c)} both factors are connected, therefore $f_1^{-1}(\nu)$ is connected too.
	
	The fiber of $f_2:X''_4(C)\rightarrow \widetilde{G_+}$ at $\nu\in \widetilde{G_+}$ is $X_2'(\nu)\times X_2(\nu^{-1}C)$, which, by the same arguments as before, is connected.
	
	By Lemma~\ref{lm:product_submersion} $f_{n-2 \mid X''_n(C)}$ is a submersion, thus it satisfies the path lifting property. Moreover, both images $\mathcal{J}_0$ and $\widetilde{G_+}$ are connected, thus $X_n''(C)$ is connected for $n=3,4$. 
	
	For $n>4$, we apply the same argument to $f_{n-2}:X''_n(C)\rightarrow \widetilde{G_+}$. The fiber at $\nu \in \widetilde{G_+}$ is $X'_{n-2}(\nu)\times X_2(\nu^{-1}C)$, which by Remark~\ref{rk:inclusion_x(c)} and the induction hypothesis is connected. 
\end{proof}

\begin{remark} The technique to compute connected components of $X'_3(\pm Id)$ was used in~\cite{Goldman1} to compute the connected components of $W(\Sigma_{0,3})\subset \mathrm{hom}(\pi_1(\Sigma_{0,3}, \mathrm{PSL}(2,\mathbb{R})))$, where $\Sigma_{0,3}$ is the three-holed sphere and $W(\Sigma_{0,3})$ is the subset of non-commuting hyperbolic representations. The three components of $W(\Sigma_{0,3})$ are distnguished as the fibers of a relative Euler class, $e^{-1}(n)$, $n=-1,0,1$. Given $\phi \in W(\Sigma_{0,3})$, if $\chi_\phi \in (2,\infty)^3$, then $\phi \in e^{-1}(0)$. Otherwise, if $\chi_\phi \in (2, \infty)^2\times (-\infty, -2)$, then $\phi \in e^{-1}(\pm 1)$, where the components $e^{-1}(-1)$ and $e^{-1}(1)$ are interchanged if $\phi$ is conjugated by an element of $\mathrm{PGL}(2,\mathbb{R})$ not in $\mathrm{PSL}(2, \mathbb{R})$.
	
	In the non-orientable case $N_3$, by Proposition~\ref{pp:x'3_connected} we can assume that any representation is hyperbolic and we can cut out the Möbius strips in order to obtain a representation \allowbreak  $Q(\phi) \in \mathrm{hom}(\pi_1(\Sigma_{0,3}), \mathrm{PSL}(2,\mathbb{R}))$. As the trace is real, by conjugation, we can also assume it is actually in $\mathrm{PSL}(2,\mathbb{R})$ and compute in which component of $W(\Sigma_{0,3})$ it is (this component is not well defined). When $\chi_{Q\phi}\in (2,\infty)^2\times(-\infty,-2)$ it can be either in $e^{-1}(1)$ or $e^{-1}(-1)$ and  we can pass from one component to the other by conjugating by the element $[\mathrm{diag}(i,-i)]\in \mathrm{PSL}(2,\mathbb{C})$.
	
	This hints that a possible approach could have been closer to the $\mathrm{PSL}(2,\mathbb{R})$ case worked out in \cite{Palesi1}, where paths of representations joining $e^{-1}(n)$ and $e^{-1}(n+2)$ are constructed.
\end{remark}

\subsection{The rest of the connected components}

When studying connected components which are not orientation type preserving we have to take into account that $\pi_0(G)=\mathbb{Z}_2$. We define the \emph{first Stiefel-Whitney class} of a representation $\phi \in \mathrm{hom}(\pi_1(N_k), G)$ as the element $w_1(\phi)\in \mathrm{hom}(\pi_1(N_k), \pi_0(G))$ obtained by postcomposing the representation with the map $G\rightarrow \pi_0(G)$. Thus, $w_1(\phi)$ can be seen as an element of $H^1(N_k, \mathbb{Z}_2)\cong \mathbb{Z}_2^k$. For instance, the first Stiefel-Whitney class of a orientation type preserving representation is $w_1(\phi)=(1,\dots, 1)$.

In Section~\ref{section:sq_map} we defined the square maps $Q:G\rightarrow \widetilde{G_+}$ and to this point we have been interested in its restriction to $G_-$. In this section we focus on the restriction to $G_+$, which we will denote $Q_+: G_+ \rightarrow \widetilde{G_+}$. We will prove that the map $Q_+$ have really nice properties, which will allow us to use path-lifting arguments in order to prove that each fiber $w_1^{-1}(\epsilon)$, $\epsilon\in \mathbb{Z}_2^k$ has two connected components.

\begin{proposition}
\label{pp:fiber_q2}
\begin{enumerate}
\item For any $B\in \widetilde{G_+}\setminus\mathrm{tr}^{-1}(-2)$, there is a unique $A\in G_+$ such that $Q_+(A)=B$, given by
$$
A=\left[\frac{B+Id}{\sqrt{\mathrm{tr}B+2}} \right].
$$

\item The fiber $Q_+^{-1}(-Id)$ is the set $\mathrm{Ad}(G_+)\left( \begin{smallmatrix}
0 & -1 \\
1 & 0
\end{smallmatrix} \right)$.
\item For $B=\left( \begin{smallmatrix}
-1 & \tau \\
0 & -1
\end{smallmatrix} \right)$,
the fiber is empty in $G_+$.
\end{enumerate}
\end{proposition}

\begin{proof}
The proof is analogous to the one of the $\mathrm{PSL}(2,\mathbb{R})$-case which can be found in \cite{Palesi1}. Given any matrix $A\in \mathrm{M}_{2\times 2}(\mathbb{R})$, $\mathrm{tr}(A^2)=\mathrm{tr}(A)^2-2\ \mathrm{det}(A)$. Thus, the first assertion is due to the previous formula and the Cayley-Hamilton theorem. Regarding the second one,  it is a straightforward computation that for any $A\in G_+$ such that $\mathrm{tr}(A)=0$, $A^2=-Id$. This also proves the last assertion.
\end{proof}

Proposition~\ref{pp:fiber_q2} (1.) shows that the map $Q_+^{-1}:\widetilde{G_+}\setminus \mathrm{tr}^{-1}(-2)$ is well-defined and it is smooth. Thus, paths can always (and uniquely) be lifted as long as they avoid $-Id$. This shows that the statements made in Subsection~\ref{subsection:orientation_preserving} for the orientation type preserving components can be done in this context for the rest of components with small modifications.

\begin{theorem}
The representation variety $\mathrm{hom}(\pi_1(N_k), G)$ has $2^{k+1}$ connected components.
\end{theorem}

\subsection*{Acknowledgements}
The author would like to thank his PhD advisor, Joan Porti, for his helpful comments and guidance during the writing of this paper.

\bibliographystyle{plain}
\bibliography{ConnectedComponents}

\noindent \textsc{Departament de Matem\`atiques, Universitat Aut\`onoma de Barcelona, 
	08193 Cerdanyola del Vall\`es}

\noindent \textsf{jduran@mat.uab.cat}
\end{document}